\documentclass[11pt]{article}

\usepackage[T1]{fontenc}
\usepackage{lmodern}
\usepackage[a4paper,margin=1.02in]{geometry}
\usepackage{amsmath,amssymb,amsthm,mathtools,mathrsfs}
\usepackage{booktabs,tabularx,array,float}
\usepackage{enumitem}
\usepackage{microtype}
\usepackage{xcolor}
\usepackage[colorlinks=true,linkcolor=blue!50!black,
  citecolor=blue!50!black,urlcolor=blue!50!black]{hyperref}
\usepackage[nameinlink,capitalise]{cleveref}

\hypersetup{%
  pdftitle={Second-pole wall periods for Witten zeta functions in the classical families},
  pdfauthor={Jonas Matuzas},
  pdfsubject={Witten zeta functions, second poles, Selberg integrals, classical root systems},
  pdfkeywords={Witten zeta function, second pole, Selberg integral,
  Dotsenko--Fateev integral, root system, wall period}
}

\allowdisplaybreaks
\setlist[itemize]{leftmargin=1.5em,itemsep=2pt,topsep=4pt}
\setlist[enumerate]{leftmargin=1.8em,itemsep=2pt,topsep=4pt}

\newtheorem{theorem}{Theorem}[section]
\newtheorem{proposition}[theorem]{Proposition}
\newtheorem{lemma}[theorem]{Lemma}

\newcommand{\Q}{\mathbb Q}
\newcommand{\Z}{\mathbb Z}
\newcommand{\dd}{\,\mathrm d}
\newcommand{\Res}{\operatorname*{Res}}
\newcommand{\Per}{\mathcal P}
\newcommand{\DF}{\mathcal S}
\newcommand{\ord}{\operatorname{ord}}

\newcommand{\SecondPoleQ}{\frac{r-1}{N-1}}
\newcommand{\SecondPoleQA}{\frac{2}{r+2}}
\newcommand{\SecondPoleQBC}{\frac{1}{r+1}}
\newcommand{\SecondPoleQD}{\frac{r-1}{r(r-1)-1}}
\newcommand{\SecondPoleQBTwo}{\frac13}
\newcommand{\SecondPoleQGTwo}{\frac15}
\newcommand{\SecondPoleResidue}[3]{\frac{\zeta_{\mathrm R}\!\left(#1\right)}{#2}\left(#3\right)}
\newcommand{\SecondPoleWallMeasure}{\mathrm d u}
\newcommand{\SecondPoleFfourValue}{-0.13729466454027363757\ldots}
\newcommand{\SecondPoleDfiveValue}{-1.76220785978119615132\ldots}
\newcommand{\SecondPoleFfourOrbitRatio}{2^{10/23}}

\title{Second-pole wall periods for Witten zeta functions\\
in the classical families}
\author{Jonas Matuzas\\
\href{mailto:jonas.matuzas@gmail.com}{\nolinkurl{jonas.matuzas@gmail.com}}}
\date{}

\begin{document}
\maketitle

\begin{abstract}
Let $\Phi$ be an irreducible reduced crystallographic root system of rank $r$,
and let $N$ be the number of positive coroots.  A previous theorem
\cite[Theorem~3.3]{MatuzasSecondPole} identifies the first pole below $2/h$ as
$q_2=\SecondPoleQ$ and expresses its residue as a sum of periods attached to the
simple walls of the dominant chamber.  We evaluate that wall-period sum for
all four classical families.  The identity $q_2(N-1)=r-1$ removes the radial
variable from every wall integral.  The resulting projective integrals reduce
to mixed Dotsenko--Fateev chambers in type $A_r$, to those chambers together
with a Selberg endpoint in types $B_r$ and $C_r$, and to a single chamber
family in type $D_r$.  The chamber recurrences give explicit sine weights in
types $A$, $B$, and $C$.  In type $D$, a terminating basic-hypergeometric sum
at a root of unity reduces to a finite cyclotomic product.  We also show that
the corresponding exceptional wall restrictions are not reflection
arrangements, so this particular reduction does not extend directly to the
exceptional types.
\end{abstract}

\noindent\textbf{Generative-AI disclosure and author responsibility.}
This work was produced using OpenAI's ChatGPT 5.6 Pro.  The author directed
and audited the work throughout.  Jonas Matuzas takes full responsibility for
the mathematics and the final text.



\section{Introduction}

For a complex simple Lie algebra, or equivalently its simply connected compact
form, the Witten zeta function is
\[
 \zeta_\Phi(s)=\sum_{\lambda\in P_+}(\dim V_\lambda)^{-s},
\]
where $P_+$ is the set of dominant integral weights.  Its abscissa of
convergence is $2/h$, where $h$ is the Coxeter number
\cite{LarsenLubotzky2008,HasaStasinski2019}, and the residue at that pole has a
uniform Macdonald--Mehta--Opdam evaluation \cite{MatuzasLeading2026}.

Theorem~3.3 of \cite{MatuzasSecondPole} identifies the first distinct pole below
$2/h$.  It is simple, occurs at
\[
 q_2=\frac{r-1}{N-1},
\]
and has residue
\[
 \Res_{s=q_2}\xi_\Phi(s)
 =\frac{\zeta_{\mathrm R}(q_2)}{N-1}
   \sum_{i=1}^{r}\Per_i(q_2),
\]
where the $\Per_i(q_2)$ are finite positive periods associated with the simple
walls.  The problem addressed here is to evaluate this sum uniformly in the
rank for the classical families.

\paragraph{Main argument.}
On a simple wall, exactly one positive coroot vanishes identically.  After that
factor is removed, the restricted product has degree $N-1$.  Since
$q_2(N-1)=r-1$, its inverse $q_2$-power has degree equal to minus the dimension
of the wall.  Radial coordinates therefore separate off a factor $dt/t$.  The
substantive part of the proof is then to identify the remaining projective
integral and its measure in each family.  In type $A_r$, scaling by the fused
coordinate gives a mixed Dotsenko--Fateev chamber.  In types $B_r$ and $C_r$,
the substitution $t=x^2$ gives a mixed chamber for the collision walls and a
separate Selberg integral for the coordinate wall.  In type $D_r$, the same
substitution treats the chain and fork walls, after which a finite chamber sum
remains.

The resulting formulas are stated in
\cref{thm:type-a,thm:type-bc,thm:type-d} and summarized in
Table~\ref{tab:families}.

\begin{table}[H]
\centering
\small
\begin{tabularx}{\textwidth}{@{}lcl>{\raggedright\arraybackslash}X>{\raggedright\arraybackslash}X@{}}
\toprule
family & range & $q_2$ & wall periods & field of ratios \\
\midrule
$A_r$ & $r\ge3$ & $\SecondPoleQA$
 & one endpoint gamma product times a sine weight
 & $\mathbf Q(\zeta_{r+2})^+$ \\
$B_r$ & $r\ge2$ & $\SecondPoleQBC$
 & collision sine weights plus one Selberg boundary term
 & $\mathbf Q(\zeta_{4(r+1)})^+$ for collision ratios \\
$C_r$ & $r\ge2$ & $\SecondPoleQBC$
 & collision sine weights plus one Selberg boundary term
 & $\mathbf Q(\zeta_{4(r+1)})^+$ for collision ratios \\
$D_r$ & $r\ge4$ & $\SecondPoleQD$
 & one chamber recurrence and an explicit finite product
 & $\mathbf Q(\zeta_{r(r-1)-1})^+$ \\
\bottomrule
\end{tabularx}
\caption{Wall periods in the four classical families.  In types $A$ and
$B/C$, $\theta=\pi q_2/2=-\pi\gamma$.}
\label{tab:families}
\end{table}

The special-function input is classical.  We use Selberg's integral, the
Dotsenko--Fateev mixed chambers and their recurrence, the $BC_n\to A$
squaring substitution, connection formulas for Selberg-type integrals, and
the terminating $q$-Chu--Vandermonde identities
\cite{Selberg1944,ForresterWarnaar2008,ForresterRains2012,Mimachi2013,GasperRahman2004}.
The contribution of this paper is the exact identification and normalization
of the Witten wall periods in each classical family, together with the
resulting residue formulas.  No new Selberg evaluation or chamber recurrence
is asserted.

Au computed the relevant Witten zeta functions in ranks two and three and
formulated the residue-shape conjectures \cite{Au2025}.  The $A_3$
specialization of \cref{thm:type-a} agrees with his published second residue.
The rank-two formulas below are included only as normalization checks; they do
not represent the difficulty of the higher-rank argument.

Section~\ref{sec:setup} fixes the normalization.  Section~\ref{sec:homogeneity}
removes the radial variable, and Section~\ref{sec:known-input} records the
classical integral formulas used in the proof.  Sections~\ref{sec:type-a}--\ref{sec:type-d}
contain the family evaluations.  The final sections give
low-rank checks, explain why the same reduction does not apply directly to the
exceptional arrangements, and state the remaining questions.

\section{Notation and wall periods}\label{sec:setup}

\subsection{Root data and normalization}

Write $\Psi=\Phi^\vee$, choose simple coroots
$\beta_1,\ldots,\beta_r$, and expand every positive coroot as
\[
 \beta=\sum_{i=1}^r b_i(\beta)\beta_i,
 \qquad b_i(\beta)\in\Z_{\ge0}.
\]
If $\lambda=\sum_i n_i\omega_i$ is dominant, put $m_i=n_i+1$.  Weyl's
dimension formula has the form
\[
 \dim V_\lambda=\frac{P_\Phi(m)}{K_\Phi},\qquad
 P_\Phi(m)=\prod_{\beta\in\Psi^+}
 \left(\sum_{i=1}^r b_i(\beta)m_i\right),
\]
where $K_\Phi=\prod_{j=1}^r \Gamma(e_j+1)$ and $e_j$ are the exponents of $\Phi$.
We use the normalized series
\[
 \xi_\Phi(s)=K_\Phi^{-s}\zeta_\Phi(s)
 =\sum_{m\in\mathbb N^r}P_\Phi(m)^{-s}.
\]
This normalization changes residues but not pole locations or orders.

\subsection{Simple walls}

Let $C=\{x_i>0\}$ be the open dominant cone in simple-coroot evaluation
coordinates.  The $i$th simple wall is
\[
 F_i=\{x_i=0,\ x_j>0\ (j\ne i)\}.
\]
Delete from $P_\Phi$ the unique factor that vanishes identically on $F_i$ and
denote the resulting homogeneous restricted product by $Q_i$.  Let
$\Delta_i\subset F_i$ be any positive affine section meeting each ray once.
The wall period is
\[
 \Per_i(q)=\int_{\Delta_i}Q_i(u)^{-q}\,\dd\mu_i(u),
 \qquad \dd\mu_i=\SecondPoleWallMeasure.
\]
At $q=q_2$ the value is independent of the chosen positive section.
Finiteness and the residue formula are proved in
\cite[Theorem~3.3]{MatuzasSecondPole}.

\subsection{Projectivization}

If a density on a cone of dimension $d$ is homogeneous of degree $-d$, then
in radial coordinates it has the form $dt/t$ times a density on any transverse
section.  The coefficient of $dt/t$ is the projective period.  The next
section proves that every simple-wall density has this degree at $q=q_2$.

\section{Homogeneity of the wall integrals}\label{sec:homogeneity}

The equality $q_2(N-1)=r-1$ is the common reason that the radial variable
can be removed in every family.

\begin{theorem}[Radial reduction on a simple wall]\label{thm:criticality}
Let $\Phi$ be an irreducible reduced crystallographic root system of rank
$r$, put $\Psi=\Phi^\vee$ and $N=|\Psi^+|$, and let $F_i$ be a simple wall of
the positive coroot cone.  Exactly one positive coroot, namely the defining
simple coroot, vanishes identically on the relative interior of $F_i$.
Consequently the restricted wall product has $N-1$ nonzero linear factors.
At the second-pole location
\[
 q_2=\SecondPoleQ,
\]
its inverse $q_2$-power is homogeneous of degree $-(r-1)$, equal to minus the
wall dimension.  In radial coordinates $x=ty$ the wall density factors as
$dt/t$ times a section-independent positive projective density on
$\mathbf P(F_i)\cong\mathbf P^{r-2}$.
\end{theorem}

\begin{proof}
A positive coroot has a nonnegative simple expansion
$\beta=\sum_j b_j\beta_j$.  On the relative interior of $F_i$, the coordinates
$x_j$ for $j\ne i$ are positive.  The corresponding linear form therefore
vanishes identically only if $b_j=0$ for all $j\ne i$.  Since the root system is reduced, the only possibility is the simple
coroot $\beta_i$.  Deleting it leaves $N-1$ linear factors, so
$\deg Q_i=N-1$.

Write $x=ty$ with $t>0$ and $y$ in a fixed positive section.  The wall has
dimension $r-1$, hence its affine measure contributes
$t^{r-2}\dd t\,\dd\mu(y)$.  Homogeneity gives
\[
 Q_i(ty)^{-q_2}=t^{-q_2(N-1)}Q_i(y)^{-q_2}.
\]
Using $q_2=\SecondPoleQ$, the exponent of $t$ is
$(r-2)-(r-1)=-1$.  Thus the radial factor is $\dd t/t$.  Changing the positive
section changes only the radial coordinate, so the coefficient is a
well-defined projective integral.
\end{proof}

This argument proves only the radial reduction.  Convergence follows from
\cite[Lemma~3.2]{MatuzasSecondPole}; the remaining projective integral is
identified separately in each classical family.

\section{Classical integral formulas}\label{sec:known-input}

This section records the classical formulas used in the family calculations.
The required convergence and nonvanishing conditions are checked at each
specialization.

\subsection{Selberg's integral}

For $n\ge1$ and parameters satisfying
\[
 \Re\alpha>0,\qquad \Re\beta>0,\qquad
 \Re\gamma>-
 \min\left\{\frac1n,\frac{\Re\alpha}{n-1},
                    \frac{\Re\beta}{n-1}\right\},
\]
with the terms containing $n-1$ omitted when $n=1$, set
\[
 S_n(\alpha,\beta,\gamma)=
 \int_{[0,1]^n}\prod_{i=1}^n
 t_i^{\alpha-1}(1-t_i)^{\beta-1}
 \prod_{i<j}|t_i-t_j|^{2\gamma}\,\dd t.
\]
Selberg's evaluation is
\[
 S_n(\alpha,\beta,\gamma)=
 \prod_{j=0}^{n-1}
 \frac{\Gamma(\alpha+j\gamma)\Gamma(\beta+j\gamma)
       \Gamma(1+(j+1)\gamma)}
      {\Gamma(\alpha+\beta+(n+j-1)\gamma)\Gamma(1+\gamma)}.
\]
These are equations (1.1)--(1.2) in the published Forrester--Warnaar survey
\cite{ForresterWarnaar2008}.  Every specialized parameter tuple below
satisfies these inequalities; for type $D$ we also record positivity of all
gamma arguments in the endpoint product.

\subsection{Mixed Dotsenko--Fateev chambers}

For $0\le p\le n$, put
$D_{n,p}=(0,1)^p\times(1,\infty)^{n-p}$ and define the unordered mixed
chamber
\[
 \DF_{n,p}(\alpha,\beta,\gamma)=
 \int_{D_{n,p}}
 \prod_i t_i^{\alpha-1}|1-t_i|^{\beta-1}
 \prod_{i<j}|t_i-t_j|^{2\gamma}\,\dd t.
\]
Whenever the adjacent chambers converge absolutely and the displayed sine
denominators are nonzero, the factorial-normalized chambers
\[
 K_{n,p}=\frac{\DF_{n,p}}{p!(n-p)!}
\]
satisfy the Dotsenko--Fateev recurrence
\[
 \frac{K_{n,p}}{K_{n,p-1}}=
 \frac{\sin\!\bigl(\pi(n-p+1)\gamma\bigr)
       \sin\!\bigl(\pi(\alpha+\beta+(n+p-2)\gamma)\bigr)}
      {\sin(\pi p\gamma)
       \sin\!\bigl(\pi(\alpha+(p-1)\gamma)\bigr)}.
\]
The chamber definition, inversion symmetry, and recurrence are published as
(2.4)--(2.6) in \cite{ForresterWarnaar2008}; they are numbered
(2.31)--(2.33) in arXiv:0710.3981v1.  The manuscript uses the published
numbering throughout.

\subsection{Squaring substitutions and connection formulas}

The substitution $t=x^2$ that folds a $BC_n$ density to an ordinary Selberg
density is classical \cite{ForresterWarnaar2008}.  In particular, the
parameter $\alpha=1/2$ arising in type $D$ below is exactly the squaring
Jacobian.  Forrester--Rains study a Selberg density with an external marked
point, split the variables at that point, and calculate the connection matrix
using Dotsenko--Fateev chambers \cite{ForresterRains2012}.  The adjacent
marked-point formulas in the Forrester--Warnaar survey are published as
(4.6)--(4.8), and appear as (3.60)--(3.62) in arXiv:0710.3981v1.

Mimachi states the connection coefficients first as finite sine expressions
and then in $q$-Racah form \cite[Theorem~2.2]{Mimachi2013}.  There the base
$q$ is fixed by the exponent of the local system, so a formal limit $q\to1$
would change the local system rather than specialize the fixed exponent used
here.  In our calculation the scaling coordinate is one of the integration
variables, and \cref{thm:criticality} removes it before the chamber recurrence
is applied.

\subsection{Terminating basic-hypergeometric identities}

The type-$D$ sum uses the two terminating $q$-Chu--Vandermonde formulas,
Appendix~II, equations (II.7) and (II.6), in the second edition of
Gasper--Rahman \cite{GasperRahman2004}; they are also DLMF
17.6.2 and 17.6.3 \cite{DLMFQChu}.  After termination, each identity is an
equality of rational functions of the base and the parameters.  It therefore
remains valid at a root of unity whenever the finitely many denominator
factors are nonzero.  The required nonvanishing check is given in
\cref{sec:type-d}; no convergence theorem for a nonterminating series is used.

\subsection{Other results used}

The leading-pole formula \cite{MatuzasLeading2026} is used only for
comparison and does not enter the proofs below.  For the exceptional
arrangements we use the theorem that a hyperplane restriction of a Weyl
arrangement is free with the largest Weyl exponent deleted
\cite[Theorem~2.14]{AbeTeraoTran2020}.

\section{Type \texorpdfstring{$A$}{A}}\label{sec:type-a}

In type $A$, a simple wall fuses two adjacent cumulative coordinates.  Using
the fused coordinate as the scale gives a mixed Dotsenko--Fateev chamber.  The
Jacobian cancels because $q_A(N-1)=r-1$, and the chamber recurrence determines
all wall periods from an endpoint value.

\begin{theorem}[Type $A_r$]\label{thm:type-a}
Let $r\ge3$, put
\[
 q_A=\SecondPoleQA,\qquad \theta_A=\frac{\pi}{r+2},
\]
and define
\[
 J_r^{\mathrm{end}}=
 \frac{\displaystyle\prod_{a=1}^{r-2}\Gamma\!\left(\frac{a}{r+2}\right)
       \displaystyle\prod_{a=3}^{r}\Gamma\!\left(\frac{a}{r+2}\right)}
 {(r-2)!\,\Gamma\!\left(\frac{r+1}{r+2}\right)^{r-2}}.
\]
The $k$th simple-wall period is
\[
 \mathcal P_{A_r,k}(q_A)=J_r^{\mathrm{end}}
 \frac{\sin(k\theta_A)\sin((k+1)\theta_A)}
      {\sin\theta_A\sin(2\theta_A)},\qquad 1\le k\le r.
\]
Hence
\[
 \sum_{k=1}^{r}\mathcal P_{A_r,k}(q_A)
 =J_r^{\mathrm{end}}\frac{r+2}{4\sin^2\theta_A}.
\]
Combining the wall sum with \cite[Theorem~3.3]{MatuzasSecondPole} gives
\[
 \operatorname*{Res}_{s=q_A}\xi_{A_r}(s)
 =\SecondPoleResidue{q_A}{(r-1)(r+2)/2}
  {J_r^{\mathrm{end}}(r+2)/(4\sin^2\theta_A)}.
\]
\end{theorem}

\begin{proof}
\subsection{Reduction to a mixed chamber}

Use cumulative coordinates
\[
 0=y_0<y_1<\cdots<y_r=1,
\]
so that the root product is the Vandermonde.  On the $k$th wall,
$y_{k-1}=y_k$; after deleting that vanishing difference, the fused point has
doubled incidence with all other points.  Denote the resulting period by
$J_{r,k}$.

For $2\le k\le r-1$, put $t=y_{k-1}$ and divide every free point by $t$.
There are $p=k-2$ points in $(0,1)$ and $r-k$ points in $(1,\infty)$.  The
Jacobian of the inverse change of variables is $t^{r-1}$.  The wall product has
\[
 \binom r2+(r-1)=\frac{(r-1)(r+2)}2
\]
factors.  At $q_A=\SecondPoleQA$, the product contributes $t^{-(r-1)}$,
so the Jacobian cancels.  The remaining density is
\[
 \prod_i x_i^{-q_A}|1-x_i|^{-2q_A}
 \prod_{i<j}|x_i-x_j|^{-q_A}.
\]
Hence
\[
 J_{r,k}=
 \frac{\DF_{r-2,k-2}(1-q_A,1-2q_A,-q_A/2)}
      {(k-2)!(r-k)!}.
\]
The two factorials remove the orderings within the intervals.  If a
collision block of $m\ge2$ ordinary points does not contain the fused point,
its singular degree is $\binom m2q_A<m-1$.  If the block contains the fused
point, its degree is
\[
 \left(\binom m2+m-1\right)q_A
 =(m-1)\frac{m+2}{r+2}<m-1
\]
for every proper block $m<r$.  Thus the mixed chambers converge.

\subsection{Recurrence and endpoint evaluation}

Insert
\[
 n=r-2,\qquad \alpha=1-q_A,\qquad
 \beta=1-2q_A,\qquad \gamma=-q_A/2
\]
into the chamber recurrence.  The denominator sines are nonzero for
$3\le k\le r$.  After factorial normalization, the four sine arguments
reduce to complementary multiples of $\theta_A$, giving
\[
 \frac{J_{r,k}}{J_{r,k-1}}
 =\frac{\sin((k+1)\theta_A)}{\sin((k-1)\theta_A)}.
\]
This telescopes.  The reflection $y\mapsto1-y$ identifies the two endpoints.  The Selberg
evaluation of either endpoint gives $J_r^{\mathrm{end}}$ in
\cref{thm:type-a}.  At the specialized parameters, the numerator factor
$\Gamma(1+(j+1)\gamma)$ and the degree-balancing denominator factor are
identical and cancel term by term.  The remaining Selberg inequalities are
strict for every $r\ge3$.

Summing the sine weights uses
\[
 \sum_{k=1}^{r}\sin(k\theta_A)\sin((k+1)\theta_A)
 =\frac{r+2}{2}\cos\theta_A.
\]
Combining this wall sum with \cite[Theorem~3.3]{MatuzasSecondPole} gives
the residue formula in \cref{thm:type-a}.
\end{proof}

At $A_3$, the coefficient of the endpoint gamma product is
$(5+\sqrt5)/10$, as recorded in \cref{eq:a3-check}.  This is exactly
Au's published value.

\section{Types \texorpdfstring{$B$ and $C$}{B and C}}\label{sec:type-bc}

The signed-coordinate coroot product has two wall geometries.  Collision
walls fold to the same mixed-chamber problem, while the coordinate boundary
folds directly to a Selberg endpoint.  Keeping those contributions separate
is essential: $B_r$ and $C_r$ share $r,N,h$ and the pole location, but their
short/long coroot scalars differ.

\begin{theorem}[Types $B_r$ and $C_r$]\label{thm:type-bc}
Let $r\ge2$, put
\[
 q_{BC}=\SecondPoleQBC,\qquad D=2(r+1),\qquad
 \theta_{BC}=\frac{\pi}{D},\qquad n=r-2,
\]
and set $a_B=2$, $a_C=1$.  Define
\[
 G_r^{\mathrm{col}}=
 \frac{\displaystyle\prod_{m=3}^{r}\Gamma(m/D)
       \displaystyle\prod_{m=r+1}^{2r-2}\Gamma(m/D)}
      {\Gamma((2r+1)/D)^{r-2}},
\]
\[
 G_r^{\mathrm{bd}}=
 \frac{\displaystyle\prod_{m=1}^{r-2}\Gamma(m/D)
       \displaystyle\prod_{m=r+3}^{2r}\Gamma(m/D)}
      {\Gamma((2r+1)/D)^{r-2}},
\]
and, for $X\in\{B,C\}$,
\[
 A_X^{\mathrm{col}}=
 \frac{2^{-n-q_{BC}}a_X^{q_{BC}}}{n!}G_r^{\mathrm{col}},\qquad
 A_X^{\mathrm{bd}}=
 \frac{2^{-n}a_X^{-(r-1)q_{BC}}}{n!}G_r^{\mathrm{bd}}.
\]
For $1\le k\le r-1$ the collision walls satisfy
\[
 \mathcal P_{X_r,k}(q_{BC})=A_X^{\mathrm{col}}
 \frac{\sin(k\theta_{BC})\sin((k+1)\theta_{BC})}
      {\sin\theta_{BC}\sin(2\theta_{BC})},
\]
whereas the coordinate boundary has period $A_X^{\mathrm{bd}}$.  Therefore
\[
 \sum_{i=1}^{r}\mathcal P_{X_r,i}(q_{BC})
 =A_X^{\mathrm{col}}\frac{r-1}{4\sin^2\theta_{BC}}
  +A_X^{\mathrm{bd}}.
\]
Combining the wall sum with \cite[Theorem~3.3]{MatuzasSecondPole} gives
\[
 \operatorname*{Res}_{s=q_{BC}}\xi_{X_r}(s)
 =\SecondPoleResidue{q_{BC}}{r^2-1}
 {A_X^{\mathrm{col}}(r-1)/(4\sin^2\theta_{BC})+A_X^{\mathrm{bd}}}.
\]
\end{theorem}

\begin{proof}
\subsection{Restricted products and the squaring substitution}

In chamber coordinates $x_1>\cdots>x_r>0$, the positive-coroot product can be
written
\[
 \Delta_X(x)=a_X^r\prod_{i=1}^r x_i
 \prod_{1\le i<j\le r}(x_i^2-x_j^2),
 \qquad a_B=2,\quad a_C=1.
\]
There are $r-1$ collision walls $x_k=x_{k+1}$ and one coordinate wall
$x_r=0$.  On a collision wall, write the fused coordinate as $z$.  Deleting
the vanishing factor gives
\[
 2a_X^r z^3\prod_{j\ne k,k+1}x_j
 \prod_{j\ne k,k+1}(x_j^2-z^2)^2
 \prod_{\substack{i<j\\i,j\ne k,k+1}}(x_i^2-x_j^2).
\]
After projectivization by $z$ and the substitution
$t_j=(x_j/z)^2$, this is the Dotsenko--Fateev density with
\[
 n=r-2,\qquad
 \alpha=\frac{r}{2(r+1)},\qquad
 \beta=\frac{r-1}{r+1},\qquad
 \gamma=-\frac1{2(r+1)}.
\]
The $t^{\alpha-1}$ power contains the surviving coordinate factor and the
Jacobian $\dd x=\dd t/(2\sqrt t)$.  The doubled factors at the fused point
give $|1-t|^{\beta-1}$, while the squared Vandermonde gives exponent
$2\gamma$.  This identifies the collision integral with the stated mixed chamber.  All
endpoint and diagonal exponents are greater than $-1$ for $r\ge2$.  If a
cluster of $m$ variables tends jointly to infinity, the exponent governing
radial convergence is
\[
 \alpha+\beta-1+\gamma(2n-m-1)=\frac{m-r+1}{2(r+1)}<0,
 \qquad 1\le m\le n.
\]
Hence the mixed chamber converges absolutely.

On the coordinate wall the restricted product is
\[
 a_X^{r-1}\prod_{i=1}^{r-1}x_i^3
 \prod_{1\le i<j<r}(x_i^2-x_j^2).
\]
Projectivizing by $x_1$ and squaring produces the separate Selberg endpoint
encoded by $A_X^{\mathrm{bd}}$.  Its specialized Selberg parameters satisfy
the inequalities of \cref{sec:known-input}; the same gamma-factor
cancellation that occurs at the collision endpoint leaves the products in
\cref{thm:type-bc}.

\subsection{Recurrence and the wall sum}

The factorial-normalized recurrence becomes
\[
 \frac{K_p}{K_{p-1}}
 =\frac{\sin((r-p-1)\theta_{BC})}
        {\sin((r-p+1)\theta_{BC})}.
\]
Since $p=r-k-1$, moving to the next collision wall gives
\[
 \frac{\Per_{X_r,k+1}}{\Per_{X_r,k}}
 =\frac{\sin((k+2)\theta_{BC})}{\sin(k\theta_{BC})}.
\]
The denominator is nonzero throughout the stated range.  The product
telescopes to the collision law in \cref{thm:type-bc}.  At
$(r+1)\theta_{BC}=\pi/2$ its sum is
$(r-1)/(4\sin^2\theta_{BC})$.  The boundary period is then added once.
The boundary-to-collision quotient is an exact quotient of gamma products;
in general it need not be cyclotomic.
\end{proof}

\section{Type \texorpdfstring{$D$}{D}}\label{sec:type-d}

Type $D$ has no coordinate wall.  Its simple walls consist of chain
collisions and the two fork walls.  After projectivization they have the same
restricted product.  Squaring therefore gives one Dotsenko--Fateev chamber
family; the two fork walls together contribute one copy of its endpoint
chamber.

\begin{theorem}[Type $D_r$]\label{thm:type-d}
Let $r\ge4$ and put
\[
 M=r(r-1)-1,\qquad q_D=\SecondPoleQD,\qquad n=r-2,\qquad h=r-1,
\]
\[
 \theta_D=\frac{h\pi}{2M},\qquad \delta_D=\frac{\pi}{2M},\qquad
 \beta=1-2q_D,\qquad \gamma=-\frac{q_D}{2},
\]
\[
 \alpha_0=-\frac12+(r-1)q_D
 =\frac{(r-1)(r-2)+1}{2M}.
\]
Define the Selberg endpoint
\[
 K_{r,0}=\frac1{(r-2)!}\prod_{j=0}^{r-3}
 \frac{\Gamma(\alpha_0+j\gamma)\Gamma(\beta+j\gamma)
       \Gamma(1+(j+1)\gamma)}
      {\Gamma(\alpha_0+\beta+(r+j-3)\gamma)\Gamma(1+\gamma)}
\]
and the finite cyclotomic factor
\[
 \Sigma_r=\frac{\sin(2\theta_D)}
 {4\sin\delta_D\cos(h\theta_D)}
 \prod_{k=1}^{r-2}\tan(k\theta_D).
\]
Then the complete wall sum is
\[
 \sum_{i=1}^{r}\mathcal P_{D_r,i}(q_D)
 =2^{4-r-q_D}K_{r,0}\Sigma_r.
\]
Combining the wall sum with \cite[Theorem~3.3]{MatuzasSecondPole} gives
\[
 \operatorname*{Res}_{s=q_D}\xi_{D_r}(s)
 =\SecondPoleResidue{q_D}{M}
 {2^{4-r-q_D}K_{r,0}\Sigma_r}.
\]
Moreover $\Sigma_r\in\mathbf Q(\zeta_M)^+$.
\end{theorem}

\begin{proof}
\subsection{Restricted product and chamber normalization}

On a chain or fork wall, let $y$ denote the fused coordinate and let
$x_1,\ldots,x_{r-2}$ be the others.  After deleting the vanishing simple
coroot, the restricted product is
\[
 Q_D(y,x)=2y\prod_{j=1}^{r-2}(y^2-x_j^2)^2
 \prod_{1\le i<j\le r-2}(x_i^2-x_j^2).
\]
The factor $2y$ is contributed by the companion root: $e_k+e_{k+1}$ on a
chain wall and the other fork root on a fork wall.  By
\cref{thm:criticality} we may use the section $y=1$, on which this factor is
the constant $2$.  With $t_j=x_j^2$ the density is
\[
 2^{-q_D-(r-2)}
 \prod_j t_j^{-1/2}|1-t_j|^{-2q_D}
 \prod_{i<j}|t_i-t_j|^{-q_D}\,\dd t,
\]
which has
\[
 \alpha=\frac12,\qquad \beta=1-2q_D,
 \qquad \gamma=-\frac{q_D}{2}.
\]
The half-power is therefore the standard squaring Jacobian.  The endpoint exponents $-1/2$ and $-2q_D$, and the diagonal exponent $-q_D$,
are all greater than $-1$ for $r\ge4$.  If $m$ variables tend jointly to
infinity, the exponent governing radial convergence is
\[
 \alpha+\beta-1+\gamma(2n-m-1)
 =\frac{(r-1)(m-r+1)-1}{2M}<0,
 \qquad 1\le m\le n.
\]
Thus every mixed chamber used below is absolutely convergent.

Let $K_{r,p}$ be the factorial-normalized mixed chambers with these
parameters.  The exact wall-lattice normalization is
\[
 \Per_{D_r,k}=2^{4-r-q_D}K_{r,r-k-1},
 \qquad 1\le k\le r-2,
\]
\[
 \Per_{D_r,r-1}=\Per_{D_r,r}=2^{3-r-q_D}K_{r,0}.
\]
The Dynkin diagram automorphism exchanging the two fork nodes preserves the
wall measure, so the two fork periods are equal.  Together they contribute one
full $p=0$ chamber, and therefore
\[
 \sum_{i=1}^r\Per_{D_r,i}
 =2^{4-r-q_D}\sum_{p=0}^{r-2}K_{r,p}.
\]

\subsection{The chamber recurrence}

The chamber recurrence specializes to
\[
 \rho_{r,p}:=\frac{K_{r,p}}{K_{r,p-1}}
 =\frac{\sin((r-p-1)(r-1)\omega)
         \sin((p(r-1)+1)\omega)}
        {\sin(p(r-1)\omega)
         \sin(((r-p+1)(r-1)-1)\omega)},
\]
where $\omega=\pi/(2M)$.  None of the denominator sines vanishes for
$1\le p\le r-2$, and every relative weight lies in $\Q(\zeta_M)^+$.  For $D_4$, triality leaves two
wall values.  The recurrence gives the central-to-outer ratio displayed in
\cref{eq:d4-ratio}.

\subsection{Evaluation of the finite chamber sum}

\begin{lemma}[Finite type-$D$ chamber sum]\label{lem:type-d-sum}
Let $n=r-2$, $h=r-1$, $M=n^2+3n+1$, and
$Q=\exp(h\pi i/M)$.  If $K_{r,p}$ is the factorial-normalized mixed chamber
used in the type-$D$ proof and
$W_{r,0}=1$, $W_{r,p}=K_{r,p}/K_{r,0}$, then
\[
 \sum_{p=0}^{n}W_{r,p}
 =(-1)^n Q^{-n(n+1)/2}
   \frac{(Q;Q)_n(1-Q^2)}{(-1;Q)_{n+3}}.
\]
Equivalently,
\[
 \sum_{p=0}^{n}W_{r,p}
 =\frac{\sin(2\theta_D)}
 {4\sin\delta_D\cos(h\theta_D)}
 \prod_{k=1}^{r-2}\tan(k\theta_D).
\]
\end{lemma}

\begin{proof}
Put $n=r-2$, $h=r-1$, and $\omega=\pi/(2M)$.  Since $rh=M+1$, the recurrence
in the preceding subsection can be rewritten as
\begin{equation}\label{eq:D-rho-sincos}
 \rho_{r,p}=
 \frac{\sin((ph+1)\omega)\cos(((p+1)h-1)\omega)}
      {\sin(ph\omega)\cos((p-1)h\omega)}.
\end{equation}
Let $\eta=e^{\pi i/M}$ and $Q=\eta^h$.  Replacing the sine and cosine factors
in \cref{eq:D-rho-sincos} by binomials gives
\[
 \rho_{r,p}=Q^{-1}
 \frac{(1-\eta Q^p)(1+\eta^{-1}Q^{p+1})}
      {(1-Q^p)(1+Q^{p-1})}.
\]
The identity $M=n^2+3n+1$ implies
\[
 \eta Q=-Q^{n+3},\qquad -\eta^{-1}Q^2=Q^{-n}.
\]
Consequently, with $(a;Q)_p=\prod_{j=0}^{p-1}(1-aQ^j)$,
\begin{equation}\label{eq:D-W-p}
 W_{r,p}=\frac{K_{r,p}}{K_{r,0}}
 =Q^{-p}\frac{(Q^{-n};Q)_p(-Q^{n+3};Q)_p}
              {(Q;Q)_p(-1;Q)_p}.
\end{equation}
Thus the required finite sum is $F(Q^{-1})$, where
\[
 F(Z)={}_2\phi_1\!\left(
 \begin{matrix}Q^{-n},-Q^{n+3}\\-1\end{matrix};Q,Z\right).
\]
The factor $(Q^{-n};Q)_p$ makes $F$ a polynomial of degree $n$.

Set
\[
 P=\frac{(Q^{-n-3};Q)_n}{(-1;Q)_n}.
\]
The two terminating $q$-Chu--Vandermonde identities
\cite[Appendix~II, (II.7) and (II.6)]{GasperRahman2004} give
\begin{equation}\label{eq:D-boundary-values}
 F(Q^{-3})=P,\qquad F(Q)=(-Q^{n+3})^nP=-Q^{-1}P.
\end{equation}
The last equality uses $n(n+3)=M-1$ and
$Q^M=(-1)^{n+1}$.

Write $F(Z)=\sum_{p=0}^n c_pZ^p$.  The quotient $c_p/c_{p-1}$ obtained from
\cref{eq:D-W-p} gives the coefficientwise identity
\begin{equation}\label{eq:D-q-difference}
 (Q^3Z-1)F(QZ)
 +\bigl[(Q^{-n}-Q^{n+3})Z+1-Q\bigr]F(Z)
 +(Q-Z)F(Z/Q)=0.
\end{equation}
Evaluate \cref{eq:D-q-difference} at $Z=Q^{-2},Q^{-1},1$.  The resulting
three linear equations determine $F(Q^{-2})$, $F(Q^{-1})$, and $F(1)$ from
the two values in \cref{eq:D-boundary-values}.  Eliminating the other two
unknowns gives
\begin{equation}\label{eq:D-elimination}
 \frac{F(Q^{-1})}{P}=
 \frac{Q^{3n}(1-Q)(1-Q^2)^2(1-Q^3)}
 {(1+Q^n)(1-Q^{2n+2})(1-Q^{2n+4})(1-Q^{n+3})}.
\end{equation}
Finally,
\[
 (Q^{-n-3};Q)_n
 =(-1)^nQ^{-n(n+7)/2}(Q^4;Q)_n.
\]
Substitution into $P$, followed by cancellation of the common finite factors
in \cref{eq:D-elimination}, yields
\[
 F(Q^{-1})=(-1)^nQ^{-n(n+1)/2}
 \frac{(Q;Q)_n(1-Q^2)}{(-1;Q)_{n+3}}.
\]
This proves the $q$-Pochhammer formula in the lemma.

All expressions above are finite.  Moreover,
$\ord(Q)=M$ for odd $n$ and $\ord(Q)=2M$ for even $n$, while
$M=n^2+3n+1>2n+4$ for $n\ge2$.  Hence no factor in $(Q;Q)_p$,
$(-1;Q)_p$, or the denominator of \cref{eq:D-elimination} vanishes.  In
particular, no convergence theorem for a nonterminating basic-hypergeometric
series is being applied at a root of unity.

To obtain the real product, write $Q=e^{2i\theta_D}$ and use
\[
 1-Q^k=-2ie^{ik\theta_D}\sin(k\theta_D),\qquad
 1+Q^k=2e^{ik\theta_D}\cos(k\theta_D).
\]
The phases cancel.  Since $r\theta_D=\pi/2+\delta_D$, the remaining real
factors give the second formula in the lemma.
\end{proof}

\subsection{Endpoint evaluation and residue}

The inversion symmetry of the mixed chamber expresses $K_{r,0}$ as the
Selberg gamma product in \cref{thm:type-d}.  The specialized parameters obey
the Selberg convergence inequalities, and every displayed gamma argument is
positive for $r\ge4$.  Multiplying that endpoint by the common coroot scalar
and by $\Sigma_r$ gives the complete wall sum.  The residue then follows from
\cite[Theorem~3.3]{MatuzasSecondPole}.
\end{proof}

\section{Low-rank and consistency checks}\label{sec:low-rank}

The rank-two cases determine the normalization conventions but do not model the
higher-rank proof: after projectivization, a rank-two wall is a point and no
integral remains.

\begin{proposition}[Rank-two normalization checks]\label{prop:rank-two}
The projective wall of a rank-two root system is a point.  In the wall order
used here,
\[
 (\mathcal P_1,\mathcal P_2)_{B_2}=(1,2^{-1/3}),\qquad
 (\mathcal P_1,\mathcal P_2)_{C_2}=(2^{-1/3},1),
\]
and
\[
 (\mathcal P_1,\mathcal P_2)_{G_2}=(18^{-1/5},2^{-1/5}).
\]
Combining the wall sum with \cite[Theorem~3.3]{MatuzasSecondPole} gives
\[
 \operatorname*{Res}_{s=\SecondPoleQBTwo}\xi_{B_2}(s)
 =\operatorname*{Res}_{s=\SecondPoleQBTwo}\xi_{C_2}(s)
 =\SecondPoleResidue{\SecondPoleQBTwo}{3}{1+2^{-1/3}},
\]
\[
 \operatorname*{Res}_{s=\SecondPoleQGTwo}\xi_{G_2}(s)
 =\SecondPoleResidue{\SecondPoleQGTwo}{5}
 {18^{-1/5}+2^{-1/5}}.
\]
\end{proposition}

The following exact values give useful checks in ranks three and four.

\begin{align}
 \frac1{4\sin^2(\pi/5)}&=\frac{5+\sqrt5}{10},\label{eq:a3-check}\\
 \frac{\sin(3\pi/8)}{\sin(\pi/8)}&=1+\sqrt2,\label{eq:b3-check}\\
 R_{D_4}&=2\cos\frac{2\pi}{11}-2\cos\frac{6\pi}{11},\label{eq:d4-ratio}\\
 R_{D_4}^5-11R_{D_4}^3+22R_{D_4}+11&=0,\label{eq:d4-polynomial}\\
 \operatorname{disc}(x^5-11x^3+22x+11)&=11^4\,43^2,\label{eq:d4-discriminant}\\
 \frac{\mathcal P_1+\mathcal P_2}{\mathcal P_3+\mathcal P_4}\bigg|_{F_4}
 &=\SecondPoleFfourOrbitRatio.\label{eq:f4-orbit-ratio}
\end{align}

The first identity is the coefficient in Au's $A_3$ formula, and the second is
the common collision ratio in $B_3$ and $C_3$.  The latter is also recovered by
direct one-dimensional Gauss--Jacobi quadrature of the wall periods.  In type
$D_4$, triality makes the three outer-wall periods equal.  The central-to-outer
ratio belongs to $\Q(\zeta_{11})^+$, has the trigonometric expression in
\cref{eq:d4-ratio}, and satisfies the displayed quintic.  For every $r\ge4$,
the finite recurrence sum in type $D_r$ agrees with the product in
\cref{lem:type-d-sum}.

The higher-rank values obtained in
\cite[Theorems~6.2 and~7.4]{MatuzasSecondPole} are
\begin{align*}
 \Res_{s=3/23}\xi_{F_4}(s)&=\SecondPoleFfourValue,\\
 \Res_{s=4/19}\xi_{D_5}(s)&=\SecondPoleDfiveValue.
\end{align*}
They agree with the present normalization.  Proposition~6.1 of the same paper
gives the $F_4$ orbit-sum ratio in \cref{eq:f4-orbit-ratio}.  This ratio also
shows that unequal scalar factors in a multiply-laced restricted product can
produce non-cyclotomic radicals.

\section{Exceptional wall restrictions}\label{sec:exceptional}

The number of restricted hyperplanes does not distinguish the classical and
exceptional cases.  For example, the simple-wall restrictions in $B_6$ and
$E_6$ both contain twenty-five hyperplanes.  The exponent multiset is the
relevant invariant.

\begin{theorem}[Exceptional wall restrictions]\label{thm:scope}
Let $\mathcal A$ be an irreducible Weyl arrangement with $N$ reflecting
hyperplanes, Coxeter number $h$, exponents
$m_1\le\cdots\le m_r=h-1$, and let $H\in\mathcal A$.  The restriction
$\mathcal A^H$ is free with exponents $m_1,\ldots,m_{r-1}$; in particular,
\[
 |\mathcal A^H|=N-h+1.
\]
For the exceptional types the restricted exponent multisets are
\[
 F_4:(1,5,7),\qquad E_6:(1,4,5,7,8),
\]
\[
 E_7:(1,5,7,9,11,13),\qquad
 E_8:(1,7,11,13,17,19,23).
\]
None is the exponent multiset of a real reflection arrangement of the same
rank.  Thus these restrictions are not reflection arrangements.  Consequently,
the classical reduction to ordinary Dotsenko--Fateev chambers does not extend
to them directly.  This statement does not exclude exact evaluations by other
methods.
\end{theorem}

\begin{proof}
The restriction theorem gives freeness with the largest Weyl exponent $h-1$
deleted \cite[Theorem~2.14]{AbeTeraoTran2020}.  Since the sum of the Weyl
exponents is $N$, the restricted arrangement contains $N-h+1$ hyperplanes.
Each exceptional list in \cref{thm:scope} contains exactly one exponent equal
to one.  A reducible reflection arrangement of positive rank would contain at
least two such exponents, so a matching reflection arrangement would have to
be irreducible.  Comparison with the finite irreducible real-reflection
exponent lists gives no match.
\end{proof}

This conclusion concerns only the reduction used for the classical families;
it does not rule out exact exceptional evaluations.  Indeed,
\cite[Theorem~6.2]{MatuzasSecondPole} evaluates the $F_4$ wall sum by a
two-orbit Dixon reduction.  Its orbit-sum ratio also shows that a uniform
cyclotomic field bound cannot hold in all multiply-laced types.

\section{Further questions}\label{sec:open}

\subsection{The remaining \texorpdfstring{$E_6$}{E6} period ratio}

The restricted arrangement yields four independent linear relations among
the six wall periods, but these relations do not determine their common
scale.

\begin{proposition}[Relations among the $E_6$ wall periods]\label{prop:e6-state}
Put $t=2\cos(2\pi/7)$, so that $t^3+t^2-2t-1=0$.  For the six simple-wall
periods, the kernel of the six-orbit Varchenko block gives
\[
 P_6=P_1,\qquad P_5=P_3,\qquad
 P_3=\frac{t(t+1)}2P_2,
\]
\[
 P_4=-\frac{3t+1}{2}P_2+(t+1)(t+2)P_1.
\]
The associated six-orbit block has rank $4$ and nullity $2$, leaving the
single ratio $\lambda_{E_6}=P_2/P_1$ undetermined.  Direct quadrature
supports eleven significant decimal digits,
\[
 \lambda_{E_6}=1.6329738485\ldots,
\]
so membership in $\mathbf Q(\zeta_{35})^+$ remains unresolved.
\end{proposition}

Resolving this ratio requires either an exact reduction or a numerical
representation with substantially higher precision.

\subsection{Composite conductors in type \texorpdfstring{$D$}{D}}

The finite product proves that $2\Sigma_r$ is a cyclotomic $S$-unit.  At prime
conductor its norm has a simple parity pattern in the computed ranks.  The
first composite case is $r=8$, for which $M=55$ and
\[
 N_{\Q(\zeta_{55})^+/\Q}(2\Sigma_8)=\frac1{11^2}.
\]
A conceptual formula for the prime-ideal exponents at general composite
$M=r(r-1)-1$ remains open.

\end{document}